\documentclass[11pt,reqno]{amsart}
\usepackage{amssymb,amsmath,amsthm,amsfonts,mathrsfs,graphicx}
\usepackage{mathtools}
\usepackage{enumerate,enumitem}
\usepackage[margin=2.5cm]{geometry}
\usepackage[colorlinks=true,linkcolor=blue]{hyperref}
\usepackage{tikz}
\usepackage{comment}

\title[Entropic solutions to the Euler-Poisson-alignment system]{Entropic solutions to the 1D pressureless Euler system with nonlocal interactions}

\author[Trevor M. Leslie]{Trevor M. Leslie}
\address[Trevor M. Leslie]{\newline Department of Applied Mathematics, Illinois Institute of Technology, Chicago IL 60616, USA}
\email{tleslie@illinoistech.edu}

\author[Changhui Tan]{Changhui Tan}
\address[Changhui Tan]{\newline Department of Mathematics, University of South Carolina, Columbia SC 29208, USA}
\email{tan@math.sc.edu}

\thanks{\textit{Acknowledgment.} 
	TL is partially supported by NSF grant DMS-2408585.  CT is partially supported by NSF grant DMS-2238219.
}

\makeatletter
\@namedef{subjclassname@2020}{\textup{2020} Mathematics Subject Classification}
\makeatother
\subjclass[2020]{35Q35, 35L65, 35D30, 35B30, 76N10, 82C22}

\keywords{Euler-Poisson-alignment system, entropic solution, scalar balance law, time-dependent flux, sticky particle dynamics}

\newtheorem{theorem}{Theorem}[section]
\newtheorem{lemma}[theorem]{Lemma}

\newtheorem{proposition}[theorem]{Proposition}

\theoremstyle{definition}
\newtheorem{definition}{Definition}[section]

\theoremstyle{remark}
\newtheorem{remark}{Remark}

\def\R{\mathbb{R}}
\def\dd{\mathrm{d}}
\def\pa{\partial}

\def\loc{\mathrm{loc}}
\def\F{\mathrm{F}}
\def\A{\mathsf{A}}
\def\tM{\widetilde{M}}

\newcommand{\sgn}[1]{\mathrm{sgn}(#1)}

\DeclareMathOperator{\supp}{supp}

\DeclareMathOperator{\Lip}{Lip}

\def\a{\alpha}

\def\s{\sigma}

\begin{document}
	\allowdisplaybreaks
	
	\begin{abstract}
		We study weak solutions of the one-dimensional pressureless Euler-Poisson-alignment system. When smooth solutions develop singularities, distributional weak solutions are not unique. We introduce an entropy-based selection principle via an associated scalar balance law with time-dependent flux and establish global well-posedness for its entropy solutions. The resulting entropic solution yields a uniquely selected weak solution of the Euler-Poisson-alignment system. In the attractive regime, it is compatible with sticky particle dynamics, while in the repulsive regime atomic states may disperse, revealing a fundamental qualitative difference between the two cases.
	\end{abstract}
	
	\maketitle 
	
	\section{Introduction}
	
	We consider the one-dimensional pressureless Euler system with nonlocal interaction forces:
	\begin{equation}\label{eq:PE}
		\begin{cases}
			\,\,\pa_t \rho +\pa_x(\rho u)=0, \qquad x\in\R, \quad t\geq0,\\
			\,\,\pa_t (\rho u) +\pa_x (\rho u^2) = \rho\, \F,
		\end{cases}
	\end{equation}
	subject to the initial data
	\[ \rho(x,0) = \rho^0(x), \quad u(x,0) = u^0(x).\]
	Here $\rho$ represents the density of a probability distribution, $u$ denotes the velocity, and $\F$ is the nonlocal interaction force. In collective dynamics, the interaction forces are often modeled within a \emph{three-zone} framework consisting of long-range attraction, short-range repulsion, and mid-range alignment.
	
	Attraction-repulsion effects are commonly modeled through an interaction potential
	\[
	\F^{AR} =-\pa_x\varphi,\quad \varphi(x,t) = \int_\R W(|x-y|)\rho(y,t)\dd y.
	\]
	A particularly important choice of $W$ is the Newtonian potential
	\[
	W(r) = -\kappa r.
	\]
	When $\kappa>0$, the interaction is repulsive, while $\kappa<0$ corresponds to attraction. In this case, the potential $\varphi$ satisfies the Poisson relation
	\[
	-\pa_x^2\varphi = 2\kappa \rho,
	\]
	and the system \eqref{eq:PE} is known as the \emph{Euler-Poisson equation}.
	
	The global well-posedness theory for smooth solutions of the Euler-Poisson system was established by Engelberg, Liu, and Tadmor \cite{engelberg2001critical}, who identified a \emph{critical threshold} phenomenon: subcritical initial data lead to global smooth solutions, while supercritical data give rise to finite-time singularity formation. This theory has been extended to higher dimensions \cite{tadmor2003critical} and in particular to radially symmetric settings \cite{wei2012critical,tan2021eulerian,carrillo2023existence}.  Pressure effects have also been incorporated in \cite{tadmor2008global}.
	
	Alignment effects are modeled by the nonlocal force
	\[
	\F^{Align} = \int_\R \phi(x-y)\big(u(y,t)-u(x,t)\big)\rho(y,t)\dd y,
	\]
	leading to the \emph{Euler-alignment system}. This system arises as the hydrodynamic limit of the Cucker-Smale particle dynamics \cite{cucker2007emergent} and serves as a macroscopic model for flocking behavior in biological swarms. The function $\phi$, known as the \emph{communication protocol}, quantifies the strength of pairwise alignment interactions and is typically assumed to be radially decreasing.
	
	Over the past decade, the Euler-alignment system has been studied extensively. For smooth solutions, critical threshold phenomena were established when $\phi$ is bounded \cite{tadmor2014critical,carrillo2016critical,he2017global}. Similar results hold for weakly singular but integrable communication kernels \cite{tan2020euler,leslie2020lagrangian}. A qualitatively different regime arises when $\phi$ is \emph{strongly singular}, i.e., non-integrable at the origin. In this case, the alignment force induces nonlocal dissipation and has a regularizing effect. In one dimension, global regularity for all smooth initial data away from vacuum has been shown in \cite{do2018global,kiselev2018global,shvydkoy2017eulerian,shvydkoy2017eulerian2,shvydkoy2018eulerian}, while rough initial data are instantaneously regularized \cite{danchin2019regular,leslie2018weak}.
	Further extensions include the incorporation of pressure effects \cite{choi2019global,chen2021global,tadmor2023swarming,bai2024global,bai2024global2} and the study of uni-directional flows \cite{LLST2020geometric,lear2021unidirectional,lear2022existence,li2024global}. We refer the reader to the surveys by Tadmor \cite{tadmor2021mathematics,tadmor2023swarming} and the monograph \cite{shvydkoy2021dynamics} and survey \cite{ShvydkoySurvey} by Shvydkoy for comprehensive overviews.
	
	In this work, we are interested in the regime where both attraction-repulsion and alignment forces are present. The resulting system is known as the \emph{Euler-Poisson-alignment (EPA) system}:
	\begin{equation}\label{eq:EPA}
		\begin{cases}
			\,\,\pa_t \rho +\pa_x(\rho u)=0,\\
			\,\,\pa_t (\rho u) +\pa_x (\rho u^2) = -\rho\pa_x\varphi + \displaystyle \int_\R \phi(x-y)(u(y)-u(x))\rho(x)\rho(y)\dd y,\\
			\,\, -\pa_x^2\varphi = 2\kappa\rho.
		\end{cases}
	\end{equation}
	The global well-posedness theory for smooth solutions of \eqref{eq:EPA} has been established in a number of works, including \cite{carrillo2016critical,kiselev2018global,bhatnagar2020critical2,bhatnagar2023critical,choi2025critical,luan2025euler}. In particular, when $\phi$ is bounded or weakly singular, solutions may develop singularities in finite time.
	
	Our objective is to continue solutions beyond the onset of singularities. This necessitates a theory of weak solutions. However, distributional weak solutions are not unique, and an additional \emph{selection principle} is required to single out a physically meaningful solution.
	
	As a brief aside, we note that in the context of measure-valued $\rho$ like we consider below, an explicit $\rho$-a.e. defined choice of $\partial_x \varphi$ is necessary to make sense of the equation.  The natural choice, on the basis of conservation of momentum and (as discussed below) BV calculus is 
	\begin{equation} 
		\label{e:EPAmeasure}
		\begin{cases}
			\,\,\pa_t \rho +\pa_x(\rho u)=0,\\
			\,\,\pa_t (\rho u) +\pa_x (\rho u^2) = \kappa (\mathrm{sgn}*\rho)\rho + \displaystyle \int_\R \phi(x-y)(u(y)-u(x))\rho(x)\rho(y)\dd y,
		\end{cases}
	\end{equation} 
	where $\sgn{x}$ is the signum function, defined to be equal to $1$ for positive $x$, $-1$ for negative $x$, and $0$ for $x = 0$.  Whenever we refer to the system \eqref{eq:EPA} below, the reader should assume we mean the specific (and canonical) manifestation \eqref{e:EPAmeasure}.
	
	In the absence of interaction forces, system \eqref{eq:EPA} reduces to the one-dimensional pressureless Euler equation. In this setting, several selection principles have been developed that lead to a unique solution compatible with \emph{sticky particle dynamics}; see, for instance, \cite{e1996generalized,brenier1998sticky,bouchut1998one}. These ideas have subsequently been extended to the Euler-Poisson equations \cite{nguyen2008pressureless,brenier2013sticky,carrillo2023equivalence,Tudorascuetal2025} and to the Euler-alignment system \cite{leslie2023sticky,leslie2024finite,galtung2025sticky}.  A more thorough overview on the sticky particle literature as it relates to alignment dynamics can be found in the introduction of \cite{leslie2023sticky}.
	
	In the present work, we develop a well-posedness theory for weak solutions of the EPA system \eqref{eq:EPA}. Our approach builds on the framework of Brenier and Grenier \cite{brenier1998sticky} for the pressureless Euler equations, and the extensions to the Euler-Poisson equations \cite{nguyen2008pressureless, NguyenTudorascu2015} and the Euler-alignment system \cite{leslie2023sticky}. The central idea is to associate \eqref{eq:EPA} with a scalar balance law
	\begin{equation}
		\label{eq:balancelaw}
		\pa_t M + \pa_x\big(\A(M,t)\big) = (\phi\ast M)\,\pa_x M,
	\end{equation}
	where $M$ denotes the cumulative distribution function (see \eqref{def:M}). The flux $\A$ depends on the initial data and is time-dependent. Its explicit form is given in \eqref{def:At}-\eqref{def:Agen}.
	
	We establish a global well-posedness theory for \eqref{eq:balancelaw} and show that entropy inequalities select a unique entropy solution. From this solution, we recover a corresponding weak solution of the EPA system \eqref{eq:EPA}, which we refer to as the \emph{entropic solution}. The entropy conditions imposed on $M$ serve as a selection principle that guarantees uniqueness.
	
	The scalar balance law \eqref{eq:balancelaw} exhibits two structural features that are central to the analysis: a genuinely nonlocal term induced by the alignment interaction, and an explicitly time-dependent flux arising from the Poisson force. While each of these features has been treated separately in the scalar equations associated with the Euler-alignment and Euler-Poisson systems, their coexistence in \eqref{eq:balancelaw} leads to new analytical difficulties. In particular, the entropy admissibility conditions and stability estimates available for the individual models do not directly extend to this setting and must be substantially refined.
	
	Our approach starts from the theory developed for the Euler-alignment system in \cite{leslie2023sticky} and incorporates the Poisson interaction through a time-dependent flux. We distinguish two scenarios depending on the sign of the Poisson force, characterized by different relationships between the entropic solution and corresponding agent-based sticky particle dynamics. 
	
	When the Poisson force is attractive ($\kappa<0$), we show that the entropic solution remains compatible with the sticky particle dynamics, extending results for the Euler-alignment system ($\kappa=0$) obtained in \cite{leslie2023sticky}. Verifying entropy admissibility in this setting presents additional technical challenges in the analysis needed to handle the time-dependent flux, see in particular the proof of Theorem \ref{thm:discreteattractive}.
	In contrast, when the Poisson force is repulsive ($\kappa>0$), the entropic solution is no longer compatible with sticky particle dynamics: atomic initial data may instantaneously disperse into absolutely continuous states under the effect of repulsion.
	\medskip

	The paper is organized as follows. 
	In Section~\ref{sec:balancelaw}, we derive the scalar balance law \eqref{eq:balancelaw} associated with the EPA system \eqref{eq:EPA} and introduce the corresponding entropy conditions. 
	In Section~\ref{sec:uniqueness}, we prove uniqueness and stability for entropy solutions of the scalar balance law \eqref{eq:balancelaw} with time-dependent flux.
	Section~\ref{sec:discbalance} is devoted to the analysis of the  agent-based sticky particle dynamics and their connection with the discretized scalar balance law: after defining the collision rules and establishing basic properties, we verify that the sticky particle solution resembles a weak solution of the discretized balance law and, in the attractive case $\kappa<0$, it recovers the entropy solution.
	In Section~\ref{sec:gendata}, we pass to the limit in the discrete entropy solutions to obtain existence for general initial data in the attractive regime. Finally, Section~\ref{sec:repulsive} discusses the repulsive case $\kappa>0$, where sticky particle dynamics cease to be compatible with the entropic solution.

	\section{The scalar balance law}\label{sec:balancelaw}
	In this section, we derive the scalar balance law \eqref{eq:balancelaw} associated with the Euler-Poisson-alignment system \eqref{eq:EPA} and introduce the corresponding notion of entropy solutions. The derivation is based on a reformulation of the EPA system in terms of transported quantities, which allows us to reduce the system to a single nonlocal scalar equation with a time-dependent flux. We then define an admissibility criterion via entropy inequalities and recall the associated Rankine-Hugoniot and Oleinik conditions, which will serve as the foundation for the well-posedness and selection theory developed in the subsequent sections.
	
	\subsection{Key quantities}
	We begin by deriving an alternative formulation of the EPA system that will be instrumental in obtaining the scalar balance law. Define the auxiliary quantity (introduced in~\cite{carrillo2016critical})
	\[
	e(x,t) = \pa_x u(x,t) + \phi*\rho(x,t).
	\]
	A direct calculation using \eqref{eq:EPA} gives
	\[
	\pa_te +\pa_x(eu) = 2\kappa\rho.
	\]
	From \eqref{eq:EPA}$_1$, we have
	\[
	\pa_t(t\rho) + \pa_x(t\rho u) = \rho.
	\]
	Combining the two equations above, we see that $e-2\kappa t\rho$ satisfies the continuity equation
	\[
	\pa_t(e-2\kappa t\rho) + \pa_x\big((e-2\kappa t\rho)u\big) = 0.
	\]
	Let $M$ denote the primitive of $\rho$, shifted to lie in $[-\tfrac12,\tfrac12]$:
	\begin{equation}\label{def:M}
		M(x,t) = - \frac12 + \int_{(-\infty,x]} \rho(y,t)\dd y.	
	\end{equation}
	Define a primitive $\psi$ of $e-2\kappa t \rho$ via 
	\begin{equation}\label{def:psi}
		\psi \triangleq u + \Phi\ast\rho - 2\kappa t M,\quad \text{where}\quad \Phi(x) \triangleq \int_0^x \phi(y)\dd y.
	\end{equation}
	Thanks to the shift by $-\frac12$ in \eqref{def:M}, the convolution $\Phi\ast\rho$ can be rewritten as $\phi\ast M$.  Indeed, if $\rho(t)$ is a compactly supported probability density with $\supp\rho(t)\subset[-R,R]$, then $M(\pm x,t)=\pm\tfrac12$ for $x\ge R$, and a simple integration-by-parts verifies $\Phi*\rho=\phi*M$ on $[-R,R]$.
	
	Expressed in terms of the pair $(\rho,\rho\psi)$, the EPA system \eqref{eq:EPA} becomes
	\begin{equation}\label{eq:EPA2}
		\begin{cases}
			\,\,\pa_t \rho + \pa_x (\rho u) = 0, \\
			\,\,\pa_t (\rho \psi) + \pa_x (\rho \psi u) = 0,
		\end{cases}  
	\end{equation}
	and the velocity can be recovered from $\psi$ and $M$ via the relation \eqref{def:psi}:
	\begin{equation}
		\label{eq:u}
		u = \psi - \phi\ast M + 2\kappa tM.
	\end{equation}
	
	\subsection{Derivation of the scalar balance law}
	We now give a formal derivation of the scalar balance law \eqref{eq:balancelaw} from the EPA system \eqref{eq:EPA2}-\eqref{eq:u}, assuming all functions involved are as regular as necessary. 
	
	Let $Q$ denote the following primitive of $\rho\psi$:
	\[
	Q(x,t) = \int_{(-\infty,x]} \rho \psi(y,t)\dd y.
	\]
	From \eqref{eq:EPA2}, we see that $M$ and $Q$ satisfy pure transport equations:
	\[
	\pa_t M + u\pa_x M  =  0, \qquad
	\pa_t Q + u\pa_x Q  =  0.
	\]
	
	Let $A$ be a map (depending only on $(M^0, Q^0)$) such that $Q^0 = A\circ M^0$. 
	Then, we have $Q(t) = A\circ M(t)$ for any $t\geq0$. Indeed, suppose $(x,t)$ lies on a characteristic path originating at $(x^0,0)$ and governed by the velocity field $u$. Then
	\[
	Q(x,t) = Q^0(x^0)=A(M^0(x^0))=A(M(x,t)).
	\]
	
	Together with the relation \eqref{eq:u}, we can write
	\begin{align*}
		u\pa_x M  & = \rho \psi - \rho \Phi*\rho +2\kappa t\rho M = \pa_x Q - \pa_x M\cdot (\phi\ast M) + \kappa t \pa_x(M^2)\\
		& = \pa_x \big(A(M)+\kappa tM^2\big) - \pa_x M \cdot (\phi\ast M).
	\end{align*}
	This leads to the scalar balance law \eqref{eq:balancelaw} with a time-dependent flux
	\begin{equation}\label{def:At}
		\A(m,t) \triangleq A(m) +\kappa t m^2 
	\end{equation} 
	We take $A(-\frac12) = 0$ and
	\begin{equation}
		\label{def:Agen}
		A(m) = \int_{-\frac12}^m \psi^0\circ X^0(m')\,\dd m',
		\qquad m\in (-\tfrac12, \tfrac12],
	\end{equation}
	where $X^0$ denotes the generalized inverse of $M^0$.  This $A$ has the desired property $Q^0 = A\circ M^0$ even when $\rho^0$ and $u^0$ are less regular.
	
	Note that when $\kappa=0$, the EPA system \eqref{eq:EPA} reduces to the Euler-alignment system. The theory for the corresponding scalar balance law \eqref{eq:balancelaw} has been established in \cite{leslie2023sticky}. The principal new difficulty here is the time-dependence of the flux $\A$.

	\subsection{Entropy conditions}
	We now introduce an admissibility criterion for \eqref{eq:balancelaw} via entropy inequalities, following the framework of \cite{leslie2023sticky} and adapting it to the time-dependent flux $\A(m,t)$.
	
	Let $\eta:[-\tfrac12,\tfrac12]\to \R$ be a convex and Lipschitz function. For this entropy $\eta$, we define the corresponding entropy flux $q:[-\tfrac12,\tfrac12]\times \R_+\to \R$ by
	\[
	\pa_m q(m,t) = \eta'(m)\, \pa_m \A(m,t).
	\]
	An entropy solution $M$ to \eqref{eq:balancelaw} is required to satisfy, in the sense of distributions,
	\begin{equation} \label{eq:entropy}
		\pa_t\big(\eta(M)\big) + \pa_x\big(q(M,t)\big) \le (\phi\ast M)\,\pa_x\big(\eta(M)\big).
	\end{equation}
	Equivalently, for any nonnegative test function $\zeta\in C_c^\infty(\R\times (0,T))$,
	\begin{equation} \label{eq:entropydist}
		\int_0^{T} \int_\R \Big[\eta(M) \pa_t \zeta + q(M,t)\pa_x \zeta 
		+ (\phi\ast M)\pa_x (\eta(M))\zeta \Big] \dd x \dd t \ge 0.
	\end{equation}
	The last term is well defined since $\eta(M)\in BV_{\loc}(\R)$ and is constant outside $[-R,R]$, while $\phi*M$ is bounded and continuous on $[-R,R]$. Consequently, $\pa_x(\eta(M))$ can be interpreted as a Radon measure, and the integral is understood in the usual sense (cf. \cite{leslie2023sticky} for details).
	
	\begin{definition}
			We say $M : \R \times [0,T] \to [-\frac12,\frac12]$ is an \emph{entropy solution} to the scalar balance law \eqref{eq:balancelaw} if
		\begin{itemize}
			\item the entropy inequality \eqref{eq:entropy} holds for every entropy pair $(\eta,q)$;
			\item $M(\cdot,t)$ is nondecreasing for every $t\in[0,T]$;
			\item there exists $R(T)>0$ such that $M(\pm x,t)=\pm\frac12$ for all $x\ge R(T)$ and $t\in[0,T]$.
		\end{itemize}
		We say $M : \R \times [0,+\infty) \to [-\frac12,\frac12]$ is an entropy solution if its restriction to any compact time interval $[0,T]$ is an entropy solution in the above sense.
	\end{definition}
	
	For purposes of establishing our existence and uniqueness results for general initial data, it suffices to consider the following Kruzkov entropy pairs:
	\begin{equation}
		\label{eq:Kruzkovpair}
		\eta(m) = |m - \alpha|, \qquad 
		q(m,t) = \sgn{m - \alpha}\big(\A(m,t) - \A(\alpha,t)\big), \qquad
		\forall~\alpha\in [-\tfrac12,\tfrac12].
	\end{equation}
	
	Next, we recall the Rankine-Hugoniot and Oleinik conditions associated with \eqref{eq:balancelaw}. The derivation is identical to that in \cite{leslie2023sticky}, with the only modification that the flux $\A$ depends on time; we therefore only sketch the result.
	
	Consider a shock curve $C = \{(x,t): x = s(t)\}$ with left and right states $M_\ell$ and $M_r$, and shock speed $\sigma(t) = \dot{s}(t)$. Applying the entropy inequality \eqref{eq:entropydist} to piecewise smooth solutions and passing to the limit yields
	\[
	(\sigma + \phi\ast M)[[\eta(M)]] \le [[q(M,t)]],\qquad \text{along }C,
	\]
	where $[[\cdot]]$ denotes the jump across the shock. Choosing $\eta = \mathrm{id}$ and $q = \A$ gives the Rankine-Hugoniot condition
	\begin{equation}
		\label{e:RH}
		\sigma +\phi\ast M = \frac{[[\A(M,t)]]}{[[M]]}.
	\end{equation}
	For $M_\ell<M_r$, taking $\eta(m) = (m-\theta) H(m-\theta)$ and 
	$q(m,t)=(\A(m,t)-\A(\theta,t))H(m-\theta)$ for $\theta\in(M_\ell, M_r)$ leads to the \emph{Oleinik entropy condition}
	\begin{equation}
		\label{eq:Oleinik}
		\sigma + \phi\ast M\le \frac{\A(\theta,t) - \A(M_\ell,t)}{\theta - M_\ell},
		\qquad \theta\in (M_\ell, M_r).
	\end{equation}
	The formulas \eqref{e:RH}-\eqref{eq:Oleinik} coincide with those in \cite{leslie2023sticky}, except for the explicit time-dependence of the flux $\A(\cdot,t)$.
	
	\begin{remark}
		\label{rem:entconds}
		Before moving on, we pause to describe the role of the conditions \eqref{eq:entropy}, \eqref{e:RH}, and \eqref{eq:Oleinik} in our framework.  
		In Section~\ref{sec:uniqueness}, we prove uniqueness and stability of solutions satisfying \eqref{eq:entropy} using the Kruzkov pairs \eqref{eq:Kruzkovpair}.  On the other hand, \eqref{e:RH} and \eqref{eq:Oleinik} are used in our proof of existence (for the attractive case $\kappa<0$), which is constructive.  First, in Section \ref{sec:discbalance}, we generate piecewise constant entropy solutions, and then in Section~\ref{sec:gendata} we obtain an entropy solution associated to general initial data by taking limits.  
		
		The discontinuities in our piecewise constant solutions evolve according to a system of ODEs, except at merging times, when they will obey `sticky particle' (i.e., completely inelastic) collision rules.  The system of ODEs enforces the Rankine-Hugoniot condition \eqref{e:RH} (which is emphatically not an entropy condition, but is by itself enough to verify that our piecewise constant solutions do in fact solve the scalar balance law in a weak sense).  In the attractive case $\kappa<0$, the sticky particle collision rules  enforce the Oleinik condition.
		The two conditions \eqref{e:RH} and \eqref{eq:Oleinik} together guarantee that \eqref{eq:entropy} holds for any entropy/entropy flux pair $(\eta, q)$.  
		
		A much wider family of collision rules can be used to generate piecewise constant weak solutions of the type we generate (even for the repulsive case $\kappa>0$). However, none of them are entropy solutions in the repulsive case, and in the attractive case, only the sticky particle collision rules are compatible with the entropy conditions.
	\end{remark}

	Once an entropy solution $M$ of \eqref{eq:balancelaw}  is obtained, we can uniquely recover an \emph{entropic solution} of the EPA system \eqref{eq:EPA} by setting
	\[
	\rho(x,t) = \pa_x M(x,t),\quad \rho(x,t) u(x,t) = \pa_x\big(\A(M(x,t),t)\big) - \pa_x M(x,t)\cdot(\phi\ast M)(x,t). 
	\]
	The recovery of the EPA system follows the procedure outlined in  \cite{leslie2023sticky}.  The only difference is the Poisson term, which emerges upon differentiating the BV function $M^2$.  Using the BV chain rule, we obtain $\partial_x(M^2) = 2\overline{M} \partial_x M$, where $\overline{M}(x) = \frac12 (M(x+) + M(x-)) = \frac12 \mathrm{sgn}*\rho(x)$.  We omit the remaining details and refer the interested reader to \cite{leslie2023sticky}.

	\section{Uniqueness and stability}
	\label{sec:uniqueness}
	
	In this section, we establish the uniqueness and stability of solutions of the scalar balance law \eqref{eq:balancelaw}, which we recall here for the reader's convenience:
	\begin{equation}\label{eq:M}
		\pa_t M + \pa_x(\A(M,t)) = (\phi\ast M)\pa_x M,\qquad M(x,0)=M^0(x).
	\end{equation}
	
	The theory of entropy solutions for scalar conservation laws is by now well established. Equation \eqref{eq:M}, however, presents two additional analytical challenges. First, the right-hand side contains a nonlocal term arising from the alignment interaction. Second, the flux $\A$ is explicitly time-dependent, reflecting the presence of the Poisson force. In the absence of the Poisson force, the corresponding scalar equation associated with the Euler-alignment system was studied extensively in \cite{leslie2023sticky}. The purpose of this section is to extend that theory to accommodate a time-dependent flux.
	
	To proceed, we first define the uniform-in-time Lipschitz bound
	\begin{equation}\label{def:AtLip}
		\|\A\|_{\Lip(T)} \triangleq \sup_{0\leq \tau\leq T} |\A(\cdot,\tau)|_{\Lip}.
	\end{equation}
	The following theorem is stated for a general time-dependent flux $\A$.
	\begin{theorem}
		\label{thm:M}
		Consider the scalar balance law \eqref{eq:M}.   Assume the initial condition $M^0$ is a nondecreasing function and that there exists an $R^0>0$ such that $M^0(\pm x)=\pm\frac12$ for all $x\ge R^0$. Fix $T>0$, and let $\A:[-\frac12,\frac12]\times[0,T]\to \R$ such that
		\begin{equation}\label{eq:ALip}
			\|\A\|_{\Lip(T)}<\infty,\quad\text{and}\quad \|\pa_\tau\A\|_{\Lip(T)}<\infty.
		\end{equation}
		The Cauchy problem \eqref{eq:M} admits a unique entropy solution
		\[
		M\in BV(\R\times[0,T]).
		\]
		associated to the initial data $M^0$.  Suppose, moreover, that $\tM$ is an entropy solution of 
		\[
		\pa_t \tM + \pa_x(\widetilde\A(\tM,t)) = (\phi*\tM)\pa_x \tM,\quad \tM(x,0)=\tM^0(x)
		\]
		with $\tM^0$ and $\widetilde\A$ satisfying the same assumptions as $M^0$ and $\A$, respectively. Then for any $t\in[0,T]$, the following stability estimate holds:
		\begin{equation}
			\label{eq:stability}
			\|M(\cdot,t)-\tM(\cdot,t)\|_{L^1(\R)}
			\le \|M^0-\tM^0\|_{L^1(\R)} + t \|\A - \widetilde\A\|_{\Lip(T)}.
		\end{equation}
	\end{theorem}
	
	The existence part of Theorem~\ref{thm:M} follows from a standard front-tracking scheme, see e.g. \cite{holden2015front}. For the flux $\A$ defined in \eqref{def:At}, when the Poisson force is attractive ($\kappa<0$), the resulting front-tracking approximation coincides with the sticky particle dynamics \eqref{eq:ABS}. This correspondence, however, no longer holds in the repulsive case ($\kappa>0$). We therefore treat these two cases separately in Sections~\ref{sec:gendata} and~\ref{sec:repulsive}.
	
	The remainder of this section is devoted to the uniqueness and stability assertions in Theorem~\ref{thm:M}. The proof largely follows the approach of \cite{leslie2023sticky}, relying on the \emph{doubling of variables} technique introduced by Kruzkov. We emphasize the modifications needed to accommodate the time-dependent flux. Note that the stability bound \eqref{eq:stability} implies uniqueness if we set $\tM^0=M^0$ and $\widetilde\A=\A$. 
	
	For a fixed $(\widetilde{x},\widetilde{\tau})$, consider the Kruzkov entropy pair \eqref{eq:Kruzkovpair} with $\alpha = \tM(\widetilde{x},\widetilde{\tau})$, and a test function  $\zeta(x,t) = w(x,\tau,\widetilde{x},\widetilde{\tau})$ to be specified later. The entropy inequality~\eqref{eq:entropydist}~reads
	\begin{align*}
		0 & \le  \iint |M(x,\tau) - \tM(\widetilde{x},\widetilde{\tau}))|\pa_{\tau} w(x,\tau,\widetilde{x},\widetilde{\tau})\,\dd x\,\dd \tau \\
		& \qquad + \iint \sgn{M(x,\tau) - \tM(\widetilde{x},\widetilde{\tau})}(\A(M(x,\tau),\tau)- \A(\tM(\widetilde{x},\widetilde{\tau}),\tau))\pa_x w(x,\tau,\widetilde{x},\widetilde{\tau}) \,\dd x \,\dd \tau \\
		& \qquad + \iint (\phi\ast M)(x,\tau) (\pa_x |M(x,\tau) - \tM(\widetilde{x},\widetilde{\tau})|) w(x,\tau,\widetilde{x},\widetilde{\tau}) \,\dd x \,\dd \tau.
	\end{align*}
	We repeat these manipulations with $\widetilde\A(\cdot, \widetilde{\tau})$ replacing
	$\A(\cdot,\tau)$ and the roles of  $M(x,\tau)$ and $\tM(\widetilde{x},\widetilde{\tau})$ interchanged.
	Integrating over the remaining free variables in both cases and adding
	the results, we obtain
	\begin{align}
		\label{e:Mdoubled}
		0 & \le \iiiint |M(x,\tau) - \tM(\widetilde{x},\widetilde{\tau})|(\pa_\tau w+\pa_{\widetilde{\tau}} w)(x,\tau,\widetilde{x},\widetilde{\tau})\,\dd x\,\dd \tau\,\dd \widetilde{x}\,\dd \widetilde{\tau} \\
		& \quad + \iiiint \sgn{M(x,\tau) - \tM(\widetilde{x},\widetilde{\tau})}\bigg[ \big(\A(M(x,\tau),\tau)- \A(\tM(\widetilde{x},\widetilde{\tau}),\tau)\big)\pa_x w(x,\tau,\widetilde{x},\widetilde{\tau}) \nonumber\\
		& \hspace{40mm} + \Big(\widetilde\A\big(M(x,\tau),\widetilde{\tau}\big)- \widetilde\A\big(\tM(\widetilde{x},\widetilde{\tau}),\widetilde{\tau}\big)\Big)\pa_{\widetilde{x}} w(x,\tau,\widetilde{x},\widetilde{\tau})\bigg] \,\dd x \,\dd \tau \,\dd \widetilde{x} \,\dd \widetilde{\tau} \nonumber\\
		& \quad + \iiiint w(x,\tau,\widetilde{x},\widetilde{\tau}) \bigg[ (\phi\ast M)(x,\tau)\pa_x |M(x,\tau) - \tM(\widetilde{x},\widetilde{\tau})| \nonumber \\
		& \hspace{50mm} + (\phi*\tM)(\widetilde{x},\widetilde{\tau})\pa_{\widetilde{x}} |M(x,\tau) - \tM(\widetilde{x},\widetilde{\tau})|\bigg] \,\dd x\,\dd \tau\,\dd \widetilde{x}\,\dd \widetilde{\tau}. \nonumber
	\end{align}
	
	We introduce the auxiliary variables
	\[
	x_+ = \tfrac{x+\widetilde{x}}{2}, 
	\quad 
	x_- = \tfrac{x-\widetilde{x}}{2},
	\quad 
	\tau_+ = \tfrac{\tau+\widetilde{\tau}}{2},
	\quad \text{and}\quad
	\tau_- = \tfrac{\tau-\widetilde{\tau}}{2},
	\]
	and take a test function of the form
	\[
	w(x,\tau,\widetilde{x},\widetilde{\tau}) 
	= b_\varepsilon\left( \tfrac{x-\widetilde{x}}{2} \right) b_\varepsilon\left( \tfrac{\tau-\widetilde{\tau}}{2} \right)g \left( \tfrac{x+\widetilde{x}}{2} \right) h_\delta\left( \tfrac{\tau+\widetilde{\tau}}{2} \right) 
	= b_\varepsilon(x_-)b_\varepsilon(\tau_-) g(x_+) h_\delta(\tau_+).
	\]
	Here $b_\varepsilon$ is a standard mollifier supported in $(-\varepsilon,\varepsilon)$,  approximating the Dirac delta distribution as $\varepsilon\to 0+$; the function $g$ is smooth, compactly supported, and identically $1$ on $[-R(T),R(T)]$ (where $R(T)$ is as in \eqref{eq:xbound}); and $h_\delta$ is supported in $[s-\delta, t+\delta]$,  identically $1$ on $[s,t]$, and linear on $[s-\delta, s]$ and $[t, t+\delta]$.
	
	For most of what follows, we will abbreviate $M = M(x,\tau)$, $\widetilde{M} = \widetilde{M}(\widetilde{x},\widetilde{\tau})$.  Performing the change of variables from $(x,y,t,s)$ to $(x_+, x_-, \tau_+, \tau_-)$ in the integrals in \eqref{e:Mdoubled}, we note that the bracketed part of the second term on the right side of \eqref{e:Mdoubled} can be rewritten as
	\[ 
	\Big(A_+(M)-A_+(\tM)\Big)\,\pa_{x_+} w+\Big(A_-(M)-A_-(\tM)\Big)\,
	\pa_{x_-} w,
	\]
	where 
	\[
	A_\pm(m) = A_\pm(m;\tau,\widetilde{\tau}) \triangleq \frac{\A(m,\tau)\pm \widetilde\A(m,\widetilde{\tau})}{2}.
	\]  
	Furthermore, the bracketed part in the last term of \eqref{e:Mdoubled} reads
	\[
	\frac{\phi\ast M + \phi*\tM}{2}\cdot \pa_{x_+} |M - \tM|
	+  \frac{\phi\ast M - \phi*\tM}{2} \cdot\pa_{x_-} |M-\tM|.
	\]
	Substituting the above into \eqref{e:Mdoubled} then yields
	\begin{align}
		\label{e:Mdoubled2}
		& 0 \le  \iiiint |M - \tM|\, \pa_{\tau_+} w\,  \,\dd x_+\,\dd \tau_+\,\dd x_-\,\dd \tau_-\\
		& +\iiiint \bigg[\sgn{M - \tM}(A_+(M)- A_+(\tM)) \pa_{x_+} w + \frac{\phi\ast (M + \tM)}2w\cdot \pa_{x_+} |M - \tM|\bigg] \,\dd x_+\,\dd \tau_+\,\dd x_-\,\dd \tau_-\nonumber\\
		& + \iiiint \bigg[\sgn{M - \tM}(A_-(M)- A_-(\tM))\pa_{x_-}w + \frac{\phi\ast (M - \tM)}2w\cdot \pa_{x_-} |M - \tM|  \bigg]\,\dd x_+\,\dd \tau_+\,\dd x_-\,\dd \tau_-. \nonumber
	\end{align}
	
	When taking $\varepsilon\to 0$, the first two terms in the right-hand side of \eqref{e:Mdoubled2} become
	\begin{align*}
		&\iint |M(x,\tau)-\tM(x,\tau)| g(x) h_\delta'(\tau)\,\dd x\,\dd \tau\\
		& + \iint \sgn{M(x,\tau) - \tM(x,\tau)}\big(A_+(M(x,\tau);\tau,\tau) - A_+(\tM(x,\tau);\tau,\tau)\big) g'(x)h_\delta(\tau)\, \dd x\,\dd \tau\\
		& + \iint \frac{(\phi\ast M)(x,\tau) + (\phi\ast\tM)(x,\tau)}2\,\pa_x |M(x,\tau) - \tM(x,\tau)| g(x) h_\delta(\tau)\,\dd x\,\dd \tau.
	\end{align*}
	Note that the second integral vanishes, since $g'(x)=0$ for $x\in[-R(T),R(T)]$.
	
	For the last term in the right-hand side of \eqref{e:Mdoubled2}, we need to control the partial derivative $\pa_{x_-}$, which is singular when $\varepsilon\to0$. We proceed with the following Lemma. 
	\begin{lemma}\label{lem:BVcalc}
		Let $f:[-\frac12,\frac12]\to \mathbb{R}$ be a Lipschitz function and let $M, \widetilde{M}, x_+, x_-$ be as above. Then $\sgn{M - \tM}(f(M)- f(\tM))$ as a function of $x_-$ belongs to $BV_{\loc}(\R)$, and in the sense of measures, we have
		\[
		\Big|\pa_{x_-}\Big(\sgn{M - \tM}(f(M)- f(\tM)\Big)\Big|\leq |f|_{\Lip} \Big( |\pa_xM| + |\pa_{\widetilde{x}}\tM|\Big).
		\]
	\end{lemma}
	The proof of the Lemma \ref{lem:BVcalc} is based on the BV chain rule. See \cite{leslie2023sticky} for details of the proof.
	Applying Lemma \ref{lem:BVcalc} with $f(m) = A_-(m)$ we obtain
	\begin{align*}
		\Big|\pa_{x_-}\Big(\sgn{M - \tM}(A_-(M)- A_-(\tM)\Big)\Big|& \leq |A_-|_{\Lip} \Big( |\pa_xM| + |\pa_{\widetilde{x}}\tM|\Big)
	\end{align*} 
	Similarly, when $f(m) = m$, we get
	\begin{align*}  \Big|\pa_{x_-}|M- \tM|\Big|& \leq |\pa_xM| + |\pa_{\widetilde{x}}\tM|.
	\end{align*}
	Note that
	\[
	A_-(m) = \frac{\A(m,\tau)-\A(m,\widetilde{\tau})}{2} + \frac{\A(m,\widetilde{\tau}) - \widetilde\A(m,\widetilde{\tau})}{2}.
	\]
	Then, $|A_-|_{\Lip}$ can be estimated as follows (recalling the notation \eqref{eq:ALip}):
	\begin{align}
		|A_-|_{\Lip} 
		&\leq \|\pa_\tau\A\|_{\Lip(T)}\,|\tau_-| + \frac12\|\A-\widetilde\A\|_{\Lip(T)}.\label{eq:Aminus}
	\end{align}
	Therefore, the last term in \eqref{e:Mdoubled2} can be controlled by 
	\begin{align*}
		&\iiiint \Big(\|\pa_\tau\A\|_{\Lip(T)}\,|\tau_-| + \frac12\|\A-\widetilde\A\|_{\Lip(T)} + \frac{|\phi\ast (M - \tM)|}2\Big)
		\Big( |\pa_xM| + |\pa_y\tM|\Big)w\, \,\dd x_+\,\dd \tau_+\,\dd x_-\,\dd \tau_-.
	\end{align*}
	Taking $\varepsilon\to0$, the first part vanishes, indicating that the time dependence of the flux does not affect the stability. The remaining part becomes
	\[
	\frac12\iint \Big(\|\A-\widetilde\A\|_{\Lip(T)} + \big|(\phi\ast M)(x,\tau) - (\phi\ast\tM)(x,\tau)\big|\Big)\Big(|\pa_xM(x,\tau) |+|\pa_x\tM(x,\tau)|\Big) g(x) h_\delta(\tau) \,\dd x\,\dd \tau.
	\]
	
	We now collect all the estimates above and obtain
	\begin{align*}
		0 \le & \iint |M(x,\tau)-\tM(x,\tau)| g(x) h_\delta'(\tau)\,\dd x\,\dd \tau\\
		& +  \frac12\|\A-\widetilde\A\|_{\Lip(T)}\iint \Big(|\pa_xM(x,\tau) |+|\pa_x\tM(x,\tau)|\Big) g(x) h_\delta(\tau) \dd x\,\dd \tau\\
		& + \frac12 \iint \bigg[ \phi*(M + \tM) \cdot \pa_x |M - \tM| + (\phi*|M-\tM|)\cdot |\pa_x M+\pa_x \tM| \bigg] g(x)h_\delta(\tau) \,\dd x \,\dd \tau. 
	\end{align*}
	The final step is precisely the same as in \cite{leslie2023sticky}. Since $g(x)=1$ for $x\in[-R(T),R(T)]$ and $h_\delta\to \textbf{1}_{[s,t]}$, we take $\delta\to0$ and deduce
	\begin{align}\label{e:Mdoubled3}
		& \int |M(x,t) - \tM(x,t)|\,\dd x - \int |M(x,s) - \tM(x,s)|\,\dd x \\
		& \leq\, \frac12 \|\A-\widetilde\A\|_{\Lip(T)}
		\int_s^t \int |\pa_x M| + |\pa_x \tM| \dd x\,\dd\tau \nonumber\\
		& \quad + \frac12 \int_s^t \int \bigg[ \Phi*(\pa_x M + \pa_x \tM) \cdot \pa_x|M-\tM| + |\pa_x M + \pa_x\tM| (\Phi*\pa_x|M-\tM|)\bigg] \dd x\,\dd\tau\nonumber\\
		& \leq \|\A-\widetilde\A\|_{\Lip(T)}(t-s),\nonumber
	\end{align}
	where we have used the fact that $\int|\pa_x M|\dd x=\int|\pa_x\tM|\dd x = 1$, and the integral in the penultimate line of \eqref{e:Mdoubled3} vanishes since $\Phi$ is odd.
	The stability bound \eqref{eq:stability} follows immediately, upon taking $s=0$.
	
	\begin{remark}[Regularity requirement of $\A$ in $t$]
		For the flux $\A$ defined in \eqref{def:At}, since $m\in[-\frac12,\frac12]$, the assumptions in \eqref{eq:ALip} are readily verified:
		\[
		\|\A\|_{\Lip(T)} \le |A|_{\Lip} + |\kappa| T,\quad
		\|\pa_\tau\A\|_{\Lip(T)} \le  |\kappa|.
		\]
		
		For a general time-dependent flux, the assumption $\|\pa_t \A\|_{\Lip(T)}<\infty$ is used only in the estimate \eqref{eq:Aminus}.  In fact, this requirement can be relaxed: it is sufficient to assume that $|\A(\cdot,t)|_{\Lip}$ depends continuously on $t$.
	\end{remark}
	
	\section{The sticky particle dynamics}
	\label{sec:discbalance}
	
	In this section, we focus on the following agent-based interacting system corresponding to the EPA system \eqref{eq:EPA}:
	\begin{equation}
		\label{eq:ABS}
		\frac{\dd x_i}{\dd t} = v_i,\quad
		m_i\frac{\dd v_i}{\dd t} = \F_i \triangleq \kappa\sum_{j=1}^N m_im_j \sgn{x_i-x_j} + \sum_{j=1}^N m_im_j \phi(x_j - x_i) (v_j - v_i)
	\end{equation}
	Here $(m_i,x_i(t), v_i(t))_{i=1}^N$ denote the masses (assumed time-independent), positions, and velocities of the $N$ agents (particles), and the two sums represent the Poisson force and the Cucker-Smale type alignment force. We assume the total mass is 1:
	\[ \sum_{i=1}^N m_i = 1,\]
	and we denote by $(\theta_i)_{i=0}^N$ the cumulative mass (shifted by $-\frac12$):
	\[
	\theta_i \triangleq -\frac{1}{2}+\sum_{j=1}^i m_j,\quad i=0,1,\ldots,N.
	\]

	In Section~\ref{sec:sticky}, we introduce the sticky particle collision rule supplementing \eqref{eq:ABS}. Some basic properties of the sticky particle dynamics are outlined in Section~\ref{sec:stickyprops}.  
	
	Next, we establish connections of the sticky particle dynamics to the scalar balance law \eqref{eq:balancelaw}. In particular, we consider the following discretized version of the scalar balance law:
	\begin{equation}
		\label{e:SBdiscreteintro}
		\pa_t M_N + \pa_x(\A_N(M_N,t)) = (\phi\ast M_N) \pa_x M_N.
	\end{equation} 
	Here $M_N$ is piecewise constant, $\A_N(m,t) = A_N(m) + \kappa t m^2$, and $A_N$ is piecewise linear, with its set of breakpoints containing all the values of $M_N$. We will show that 
	\begin{equation} 
		\label{e:MNform}
		M_N(x,t) = -\frac12 + \sum_{i=1}^N m_{i,N} H(x - x_{i,N}(t))
	\end{equation} 
	generates a solution of \eqref{e:SBdiscreteintro}, where  $x_{i,N}(t)$ is a solution of the sticky particle dynamics \eqref{eq:ABS}. 
	In Section~\ref{sec:discrete}, we show that \eqref{e:MNform} satisfies the Rankine-Hugoniot condition \eqref{e:RH} and therefore it is a weak solution of the discretized balance law \eqref{e:SBdiscreteintro}.
	
	In Section~\ref{sec:attractive}, we focus on the attractive case $\kappa<0$ and further verify the Oleinik entropy condition \eqref{eq:Oleinik}. Hence, \eqref{e:MNform} is the entropy solution of the discretized balance law \eqref{e:SBdiscreteintro}. 
	Notably, the argument is no longer valid in the repulsive case $\kappa>0$. Further discussion will be postponed to Section~\ref{sec:repulsive}.

	\subsection{Sticky particle collision rules}
	\label{sec:sticky}

	Let us set the following notation:
	\[ 
	J_i(t)\triangleq\big\{j\in 1, \ldots, N:x_j(t) = x_i(t)\big\}, 
	\]
	that is, $J_i(t)$ denotes the indices of all agents that occupy the same position as agent $i$ at time~$t$.  We say that $t$ is a \textit{collision time} if the cardinality of at least one of these sets changes at time $t$.  At collision times, we require two things.  First, collisions must conserve momentum:
	\begin{equation}
		\label{e:collconsmom}
		\sum_{j\in J_i(t)} m_j v_j(t+) = \sum_{j\in J_i(t)} m_j v_j(t-)
	\end{equation}
	We also ask that agents that collide should stick together for all time.  This is the `sticky particle' collision rule.  We enforce this by requiring the sets $J_i(t)$ to be nondecreasing in $t$:
	\begin{equation}
		\label{e:sticky}
		J_i(s)\subseteq J_i(t),\qquad 0\le s\le t.    
	\end{equation}

	The discrete dynamics that we focus on thus have two key ingredients: First, at all noncollisional times they obey the ODE system \eqref{eq:ABS}, and second, at collision times, they obey \eqref{e:collconsmom} and \eqref{e:sticky}.  Since at most $N-1$ collisions can occur, these requirements define a globally well-posed system.
	
	In what follows, we will make the convention that each $v_i$ is right-continuous, i.e., $v_i(t+) = v_i(t)$ for all $t\ge 0$.  We write $x_i(0) = x_i^0$ and $v_i(0-) = v_i^0$, the latter convention being made to ensure compatibility of right continuity at time zero with the collision rules.  We index the agents in increasing order from left to right according to their initial position: $x_1^0\le x_2^0\le \cdots \le x_N^0$. Since the sticky particle rule does not allow particles to cross each other, this ordering remains in place all time:
	\begin{equation}\label{eq:order}	
		x_1(t)\le x_2(t)\le \cdots \le x_N(t),\quad \forall~t\geq0.
	\end{equation}
	Finally, it will be convenient to set notation for lowest and highest indices in a given $J_i(t)$:
	\[
	i_*(t) = \min J_i(t), 
	\qquad 
	i^*(t) = \max J_i(t).
	\] 
	
	\subsection{Basic properties of the sticky particle dynamics}
	
	\label{sec:stickyprops} We now highlight a few properties of the sticky particle dynamics that will be used later.  Our first observation is that since the agents are required to follow the ODE system \eqref{eq:ABS} at noncollisional times, particles occupying the same position must have the same velocity. Therefore, we must have $v_j(t) = v_i(t)$ for all $j\in J_i(t)$, for all noncollisional times $t$.  The sticky particle collision rules guarantee that this must hold for collisional times as well, and that in fact (using \eqref{e:collconsmom}),
	\begin{equation}
		\label{eq:vstick}
		v_i(t) = \frac{\sum_{j\in J_i(t)} m_j v_j(t-)}{\sum_{j\in J_i(t)} m_j}.
	\end{equation}
	
	Next, we bound the diameter of the agents' positions and velocities.
	\begin{proposition}\label{prop:xvbound}
		Suppose $(x_i(t),v_i(t))_{i=1}^N$ satisfies the discrete dynamics defined in Section~\ref{sec:sticky}, with initial data
		\[
		\{x_i^0\}_{i=1}^N\subset [-R^0,R^0],\quad \{v_i^0\}_{i=1}^N\subset [-V^0,V^0].
		\]
		Then for any $t\geq0$, we have
		\begin{align}
			& \{v_i(t)\}_{i=1}^N\subset [-V(t),V(t)],\quad V(t) \triangleq V^0 + |\kappa| t. \label{eq:vbound} \\
			& \{x_i(t)\}_{i=1}^N\subset [-R(t),R(t)],\quad R(t) \triangleq R^0 + V^0t +\tfrac12|\kappa| t^2,\label{eq:xbound}
		\end{align}
	\end{proposition}
	\begin{proof}
		It is clear from \eqref{eq:vstick} that $\max_{1\leq i\leq N} v_i(t) \le \max_{1\leq i\leq N} v_i(t-)$ and similarly for the minimum velocity.  Therefore, it suffices to establish \eqref{eq:vbound} on noncollisional time intervals, and then \eqref{eq:xbound} follows immediately.
		
		From \eqref{eq:ABS} we have that for almost all $t\geq0$, 
		\[
		\frac{\dd}{\dd t} \max_{1\leq i\leq N} v_i(t)\leq |\kappa| \sum_{j=1}^N m_j +0 = |\kappa|,\quad\text{then}\quad 
		\max_{1\leq i\leq N} v_i(t)\leq \max_{1\leq i\leq N} v_i^0+|\kappa| t.
		\]
		The same argument applies to $\min_i v_i(t)$, yielding \eqref{eq:vbound}.
	\end{proof}

	With the notation and conventions we have introduced, the term corresponding to the Poisson force in our velocity equation can be written in a slightly more transparent and useful form, namely 
	\begin{equation}
		\label{e:Poissondiscterm}
		\sum_{j=1}^N m_j \sgn{x_i(t)-x_j(t)}
		= \theta_{i_*(t)-1} + \theta_{i^*(t)}.
	\end{equation}
	Indeed, we have (omitting time dependence for brevity)
	\begin{align*}
		\sum_{j=1}^N m_j \sgn{x_i-x_j} & = \sum_{j=1}^{i_*-1}m_j - \sum_{j=i^*+1}^N m_j = \big(\theta_{i_*-1}+\tfrac12\big) - \big(1-(\theta_{i^*}+\tfrac12)\big)
		= \theta_{i_*-1} + \theta_{i^*}.
	\end{align*}

	Another consequence of \eqref{e:Poissondiscterm} is the following lemma, which will be useful later in this section and also in Section \ref{sec:attractive}.
	\begin{lemma}
		\label{l:thetaid}
		For any $i\in \{1, \ldots, N\}$, $t\ge 0$, $k\in J_i(t)$, and $s\in [0,t]$, the following holds:
		\begin{equation} 
			\label{e:thetaid}
			\sum_{j=i_*(t)}^k m_j (\theta_{j^*(s)} + \theta_{j_*(s) - 1}) = \theta_k^2 - \theta_{i_*(t)-1}^2 + (\theta_k - \theta_{k_*(s) - 1})(\theta_{k^*(s)} - \theta_k). 
		\end{equation} 
	\end{lemma}
	\begin{proof}
		We first observe the identity
		\begin{align}
			\sum_{j\in J_\ell(s)} m_j (\theta_{j^*(s)} + \theta_{j_*(s) - 1}) & = \Big(\sum_{j\in J_\ell(s)} m_j\Big) (\theta_{\ell^*(s)} + \theta_{\ell_*(s) - 1})\nonumber\\
			& = (\theta_{\ell^*(s)} - \theta_{\ell_*(s) - 1})(\theta_{\ell^*(s)} + \theta_{\ell_*(s) - 1}) = \theta_{\ell^*(s)}^2 - \theta_{\ell_*(s) - 1}^2,\label{e:idsinglecluster}
		\end{align}
		as particles in each cluster $J_\ell(s)$ share the same $\ell_*(s)$ and $\ell^*(s)$. Next, we separate the summation in \eqref{e:thetaid} into two parts
		\[
		\sum_{j=i_*(t)}^k m_j (\theta_{j^*(s)} + \theta_{j_*(s) - 1}) = \sum_{j=i_*(t)}^{k_*(s)-1} m_j (\theta_{j^*(s)} + \theta_{j_*(s) - 1}) + \sum_{j=k_*(s)}^k m_j (\theta_{j^*(s)} + \theta_{j_*(s) - 1})
		\]
		The first part contains multiple clusters at time $s$, thanks to \eqref{e:sticky}. For each cluster, applying \eqref{e:idsinglecluster}, the telescoping sum yields
		\[
		\sum_{j=i_*(t)}^{k_*(s)-1} m_j (\theta_{j^*(s)} + \theta_{j_*(s) - 1}) = \theta_{k_*(s)-1}^2 - \theta_{i_*(t)-1}^2.
		\]
		For the second part, a computation analogous to \eqref{e:idsinglecluster} yields
		\[
		\sum_{j=k_*(s)}^k m_j (\theta_{j^*(s)} + \theta_{j_*(s) - 1}) = \Big(\sum_{j=k_*(s)}^k m_j\Big) (\theta_{k^*(s)} + \theta_{k_*(s) - 1})
		= (\theta_{k} - \theta_{k_*(s) - 1})(\theta_{k^*(s)} + \theta_{k_*(s) - 1}).
		\]
		Combining the two parts yields \eqref{e:thetaid}, as desired.
	\end{proof}
	We now introduce a key conserved quantity.  To motivate it, we start by writing the velocity equation as  
	\begin{equation}
		\label{e:vi2}
		\frac{\dd v_i}{\dd t} = \kappa( \theta_{i_*(t)-1} + \theta_{i^*(t)} ) + \sum_{j=1}^N m_j \phi(x_j - x_i) (v_j - v_i).
	\end{equation}
	It is apparent that the right side of the equation is (at non-collisional times) a perfect time derivative, which motivates the following definition:
	\begin{equation}\label{def:psiit}
		\psi_i(t) = v_i(t) - \kappa t (\theta_{i_*(t)-1}+\theta_{i^*(t)})  + \sum_{j=1}^N m_j \Phi(x_i(t) - x_j(t)),\quad i=1,\ldots,N.
	\end{equation}
	The quantities $\{\psi_i(t)\}_{i=1}^N$ are the discrete analogue of the function $\psi(x,t)$, defined in \eqref{def:psi}. Indeed, we have the following proposition.
	
	\newpage
	\begin{proposition}\label{prop:psii}
		The quantities $\psi_i$ have the following properties, for $i = 1,\ldots, N$.
		\begin{enumerate}[label = \emph{(\roman*)}]
			\item If $t$ is not a collision time, 
			\begin{equation}
				\label{eq:psii0}
				\frac{\dd}{\dd t}\psi_i(t) = 0.
			\end{equation}
			\item  If $t$ is a collision time, then
			\begin{equation}
				\sum_{j\in J_i(t)} m_j\psi_{j}(t) = \sum_{j\in J_i(t)} m_{j}\psi_{j}(t-).
			\end{equation}
			\item As a consequence of \emph{(i)} and \emph{(ii)}, we have 
			\begin{equation}\label{eq:psiit}
				\sum_{j\in J_i(t)} m_j\psi_{j}(s) = \sum_{j\in J_i(t)} m_{j}\psi_{j}(t),\qquad \text{ for all }\; 0\le s\le t.
			\end{equation}
		\end{enumerate}	
	\end{proposition}
	\begin{proof}
		Statement (i) is clear after comparing the definition \eqref{def:psiit} with \eqref{e:vi2}. As for (ii), we check the proposed identity for each of the three terms in the definition \eqref{def:psiit} of $\psi_i$.  The collision rule \eqref{e:collconsmom} takes care of the contribution from $v_i$.  The last term in  \eqref{def:psiit} is continuous in time, since the trajectories $(x_i(t))_{i=1}^N$ are continuous.  
		Therefore, it remains to check the term in the definition of $\psi_i(t)$ involving $\theta_{i_*(t)-1} + \theta_{i^*(t)}$.  We use Lemma \ref{l:thetaid} with $s = t-$ and $k = i^*(t)$. Then the second term in \eqref{e:thetaid} vanishes, and we are left with
		\begin{align*}
			\sum_{j\in J_i(t)} m_{j}\big(\theta_{j_*(t-)-1}+\theta_{j^*(t-)}\big)
			& = \theta_{i^*(t)}^2 - \theta_{i_*(t)-1}^2 = 
			\sum_{j\in J_i(t)} m_j (\theta_{j^*(t)} + \theta_{j_*(t)-1}).
		\end{align*}
	\end{proof}

	We make one final observation before moving on: By the same logic we used to conclude \eqref{eq:vstick}, the sticky particle collision rule \eqref{e:sticky} combined with Proposition \ref{prop:psii} yields
	\begin{equation}
		\psi_i(t) = \frac{\sum_{j\in J_i(t)} m_j \psi_j(s)}{\sum_{j\in J_i(t)} m_j},
		\qquad 0\le s\le t.
	\end{equation}

	\subsection{The discretized scalar balance law}\label{sec:discrete}
	We now show that solutions of the system \eqref{eq:ABS}, equipped with the collision rules \eqref{e:collconsmom} and \eqref{e:sticky}, naturally give rise to weak solutions of the discretized scalar balance law \eqref{e:SBdiscreteintro} which verify the Rankine-Hugoniot condition \eqref{e:RH}.  The statement holds in both attractive ($\kappa<0$) and repulsive ($\kappa>0$) cases.

	For the convenience of the reader, we recall the discretized version of the scalar balance law \eqref{eq:balancelaw}, equipped with discretized initial data $M_N^0$ and flux $\A_N$:
	\begin{equation}
		\label{eq:MN}
		\pa_t M_N + \pa_x(\A_N(M_N,t)) = (\phi\ast M_N) \pa_x M_N, \qquad
		M_N(x,0)=M_N^0(x).
	\end{equation}
	We assume that $M_N^0$ is piecewise constant, of the form
	\begin{equation}
		\label{eq:MN0}
		M_N^0(x) = -\frac12 + \sum_{i=1}^{N} m_i H(x - x_i^0),
	\end{equation}
	where $H$ denotes the right-continuous Heaviside function, with $H(0) =1$. Observe that
	\[
	\rho_N^0(x) = \frac{\dd}{\dd x} M_N^0(x) = \sum_{i=1}^{N} m_i \delta_0( x-x_i^0).
	\]
	Hence, $M^0_N$ represents the $N$-particle configuration $(m_i,x_i^0)_{i=1}^N$. 
	
	The initial velocities $(v_i^0)_{i=1}^N$ are encoded in the flux $\A_N$, defined as follows.
	Let $A_N$ be a continuous and piecewise linear function, with breakpoints only at $(\theta_i)_{i=1}^{N-1}$:
	\begin{equation}\label{eq:AN}
		A_N:[-\tfrac12,\tfrac12]\to~\R,\quad
		\text{$A_N$ is linear in each interval $[\theta_{i-1}, \theta_i]$, for any
			$i=1,\ldots,N$.}
	\end{equation}
	For simplicity, we set $A_N(-\frac12)=0$. We require the slope of the $i$-th line segment of $A_N$ to be $\psi_i^0$, namely
	\begin{equation}\label{eq:psi0}
		A_N'(m) \triangleq \psi_i^0 = v_i^0 + \sum_{j=1}^{N} m_j \Phi(x_i^0 - x_j^0),\quad \forall\, m\in(\theta_{i-1}, \theta_i).
	\end{equation}
	This completely determines the function $A_N$. Then, according to \eqref{def:At}, the discretized flux $\A_N$ is defined as
	\begin{equation}\label{def:AN}
		\A_N(m,t) \triangleq A_N(m)+\kappa t m^2. 
	\end{equation}
	
	We are interested in the solution which takes the same form as $M_N^0$, that is,
	\begin{equation}
		\label{eq:MNsoln}
		M_N(x,t) = -\frac12+\sum_{i=1}^{N} m_i H(x-x_i(t)).
	\end{equation}
	
	The following theorem shows that $M_N$ is a \emph{weak} solution to the discretized balance law~\eqref{eq:MN}, where $(x_i)_{i=1}^N$ solves the sticky particle dynamics \eqref{eq:ABS}.
	\begin{theorem} \label{thm:discrete}
		Let $(x_i(t),v_i(t))_{i=1}^{N}$ be the sticky particle solution of \eqref{eq:ABS} associated to the initial data $(m_i, x_i^0,v_i^0)_{i=1}^{N}$. 
		Consider the scalar balance law \eqref{eq:MN} with discrete initial data $M_N^0$ and flux $\A_N$, defined in \eqref{eq:MN0} and \eqref{def:AN}, respectively.
		Then, $M_N$ defined in \eqref{eq:MNsoln} is a weak solution of \eqref{eq:MN}. Moreover, we have
		\begin{equation}\label{e:QN}
			A_N\circ M_N (x,t) = \sum_{i=1}^{N} m_i \psi_i(t)H(x-x_i(t)).
		\end{equation}
	\end{theorem}
	
	\begin{proof}
		Since $M_N(\cdot, t)$ in \eqref{eq:MNsoln} is piecewise constant, it suffices to check that the shock discontinuities along the curves $C_i = \{(x_i(t), t):t\ge 0\}$ satisfy the Rankine-Hugoniot condition~\eqref{e:RH} with the shock speed $\s_i(t)=v_i(t)$.
		
		Fix a point $(x_i(t), t)$ on $C_i$. By definition \eqref{eq:MNsoln}, we get 
		\[
		M_N(x_i(t)-,t) = \theta_{i_*(t)-1},\quad M_N(x_i(t)+,t)=M_N(x_i(t),t)=\theta_{i^*(t)}.
		\]
		We denote the jump of a function $f$ across $C_i$ by $[[f]] = f(x_i(t)+) - f(x_i(t)-)$. Thus
		\[
		[[M_N(\cdot, t)]] = \theta_{i^*(t)}-\theta_{i_*(t)-1}=\sum_{j\in J_i(t)} m_j,
		\]
		and from \eqref{eq:psi0}, \eqref{def:AN}, we have
		\begin{align*}
			[[ \A_N ( M_N(\cdot, t),t)]] & = 
			\int_{\theta_{i_*(t)-1}}^{\theta_{i^*(t)}} A_N'(m)\dd m + \kappa t \big(\theta_{i^*(t)}^2-\theta_{i_*(t)-1}^2\big)\\ 
			& = \sum_{j\in J_i(t)} m_j \psi_j^0 + \kappa t \big(\theta_{i^*(t)}+\theta_{i_*(t)-1}\big) \sum_{j\in J_i(t)} m_j.
		\end{align*}
		We may thus verify the Rankine-Hugoniot condition \eqref{e:RH}:
		\begin{align*}
			\frac{[[\A_N\circ M_N(\cdot, t)]]}{[[M_N(\cdot, t)]]} & = \frac{\sum_{j\in J_i(t)} m_j \psi_j^0}{\sum_{j\in J_i(t)} m_j} + \kappa t \big(\theta_{i^*(t)}+\theta_{i_*(t)-1}\big)\\
			& = \psi_i(t) + \kappa t \big(\theta_{i^*(t)}+\theta_{i_*(t)-1}\big) = v_i(t) + \phi\ast M_N(x_i(t),t), \quad \text{ along }C_i.
		\end{align*}

		Finally, we check \eqref{e:QN}. The equality is trivial when $x<x_1(t)$. For $x\ge x_1(t)$, let $i$ be the largest index such that $x>x_i(t)$.  Then we have $M_N(x,t)=\theta_i$, so that
		\[
		A_N(M_N(x,t))=\sum_{j=1}^i\int_{\theta_{j-1}}^{\theta_j}A_N'(m)\dd m
		=\sum_{j=1}^im_j\psi_j^0.
		\]
		The identity \eqref{eq:psiit} (together with the fact that $i = i^*(t)$, by construction) implies that
		$\sum_{j=1}^im_j\psi_j^0=\sum_{j=1}^i m_j\psi_j(t)$, which finishes the proof.
	\end{proof}

	\begin{remark}
		Our verification of the Rankine-Hugoniot condition holds not only for both $\kappa~<~0$ and $\kappa>0$, but also for a much wider range of collision rules than we have explicitly developed above.  In order for the verifications of this subsection to go through, we only need the conservation \eqref{e:collconsmom} of momentum across collisions and the well-posedness of the discrete dynamics. The only place where we leveraged the stickiness of the trajectories was our implicit use of \eqref{eq:order}, which can be restored if one is willing to relabel the particles at each collision time.  (The same is true of Proposition \ref{prop:psii}, on which we have also relied.)  
		For this reason, it is clear that \eqref{e:RH} cannot alone serve as a selection principle for our system, even if one restricts attention to discrete solutions of the form \eqref{e:MNform}.
	\end{remark}
	
	\subsection{Attractive Poisson force: Sticky particle collisions and entropy solutions}\label{sec:attractive}

	In this section, we consider the EPA system \eqref{eq:EPA} with an attractive Poisson force ($\kappa<0$). We show that the solution constructed in Section~\ref{sec:discrete} for the discretized balance law \eqref{eq:MN}, obtained via the sticky particle dynamics \eqref{eq:ABS}, is the entropy solution of \eqref{eq:MN}.

	\begin{theorem}\label{thm:discreteattractive}
		Let $\kappa<0$. Under the assumptions in Theorem \ref{thm:discrete}, the sticky particle solution $M_N$ is the entropy solution of the discretized balance law \eqref{eq:MN}.	
	\end{theorem}
	\begin{proof}
		To show $M_N$ is an entropy solution, we need to verify the Oleinik entropy condition \eqref{eq:Oleinik}, that is,
		\begin{align*}
			v_i(t) + \phi\ast M_N(x_i(t),t)  \le \frac{\A_N(\theta,t) - \A_N(\theta_{i_*(t)-1},t)}{\theta - \theta_{i_*(t)-1}} = \frac{A_N(\theta) - A_N(\theta_{i_*(t)-1})}{\theta - \theta_{i_*(t)-1}} + \kappa t (\theta+\theta_{i_*(t)-1}),
		\end{align*}
		for any $\theta\in (\theta_{i_*(t)-1}, \theta_{i^*(t)})$.
		For fixed $k\in J_i(t)$ and $\theta\in (\theta_{k-1}, \theta_k]$, we have (using \eqref{eq:AN}-\eqref{eq:psi0}) that
		\[
		\frac{A_N(\theta) - A_N(\theta_{i_*(t)-1})}{\theta - \theta_{i_*(t)-1}} = \frac{\sum_{j=i_*(t)}^{k-1} m_j\psi_j^0 + (\theta-\theta_{k-1})\psi_k^0}{\theta - \theta_{i_*(t)-1}}.
		\]
		On the other hand, by definition \eqref{def:psiit} we know
		\[
		v_i(t) + \phi\ast M_N(x_i(t),t) = \psi_i(t) + \kappa t(\theta_{i_*(t)-1}+\theta_{i^*(t)}).
		\]
		Therefore, it remains to verify that
		\begin{equation}\label{eq:verify}
			\psi_i(t) \leq \frac{\sum_{j=i_*(t)}^{k-1} m_j\psi_j^0 + (\theta-\theta_{k-1})\psi_k^0}{\theta - \theta_{i_*(t)-1}} - \kappa t (\theta_{i^*(t)}-\theta),
			\qquad k\in J_i(t),\, \theta\in(\theta_{k-1}, \theta_k],
		\end{equation}
		or equivalently
		\begin{equation}\label{eq:verify3}
			g(\theta)\triangleq\psi_i(t)(\theta - \theta_{i_*(t)-1}) + \kappa t (\theta_{i^*(t)}-\theta)(\theta - \theta_{i_*(t)-1}) - (\theta-\theta_{k-1})\psi_k^0  - \sum_{j=i_*(t)}^{k-1} m_j\psi_j^0 \leq 0.
		\end{equation}
		It is easy to check that $g$ is continuous in $(-\frac12,\frac12]$, and since $\kappa<0$ we have
		\[
		g''(\theta) = -2\kappa t >0,\qquad k\in J_i(t),\, \theta\in(\theta_{k-1}, \theta_k].
		\]
		Therefore, $g$ has to attain its maximum at the end points, namely
		\[
		g(\theta)\leq \max\{g(\theta_{k-1}),g(\theta_k)\}.
		\]
		Hence, it suffices to show that \eqref{eq:verify3} (or \eqref{eq:verify}) holds at the endpoints $\theta=\theta_k$ for $k\in J_i(t)$, which is a direct consequence of the following Lemma~\ref{lem:barycentric}.
	\end{proof}

	\begin{lemma}\label{lem:barycentric}
		Given $t\ge 0$ and $k\in \{i_*(t), \ldots, i^*(t)\}$, we have
		\begin{equation}
			\label{eq:barycentric}
			\psi_i(t) \leq \frac{\sum_{j=i_*(t)}^k m_j \psi_j^0}{\sum_{j=i_*(t)}^k m_j} - \kappa t\big(\theta_{i^*(t)} - \theta_k\big) .
		\end{equation}  
	\end{lemma}
	
	\begin{proof} 
		Let us write \eqref{eq:barycentric} in the form analogous to \eqref{eq:verify3}:
		\begin{equation} 
			\label{e:baryequiv}
			\sum_{j=i_*(t)}^k m_j \psi_j(t) + \kappa t (\theta_k - \theta_{i_*(t)-1})(\theta_{i^*(t)} - \theta_k) \le \sum_{j=i_*(t)}^k m_j \psi_j^0,
		\end{equation} 
		where we have used the fact that $\psi_j(t) = \psi_i(t)$ for $j=i_*(t), \ldots, k$.
		Denote 
		\[
		f(s)\triangleq \sum_{j=i_*(t)}^k m_j \psi_j(s) + \kappa s (\theta_k - \theta_{k_*(s)-1})(\theta_{k^*(s)}-\theta_k).
		\]
		Then, noting that $k_*(t) = i_*(t)$, $k^*(t) = i^*(t)$, we see that the left side of \eqref{e:baryequiv} is equal to $f(t)$, while the right side is $f(0)$. Consequently, in order to prove \eqref{e:baryequiv}, it suffices to show that $f$ is nonincreasing on $[0,t]$. 
		
		If $s$ is not a collision time, we use \eqref{eq:psii0} and $\kappa<0$ to verify that
		\[f'(s)=\kappa(\theta_k - \theta_{k_*(s)-1})(\theta_{k^*(s)}-\theta_k)\leq0.\]
		It remains to show that $f(s)\le f(s-)$ when $s$ is a collision time.
		Using the identity \eqref{e:thetaid} and the definition of $\psi_j$, we see that 
		\begin{align*}
			f(s) & = \sum_{j=i_*(t)}^k m_j \Big(\psi_j(s) + \kappa s(\theta_{j^*(s)} + \theta_{j_*(s)-1})\Big) - \kappa s \big( \theta_k^2 - \theta_{i_*(t) - 1}^2 \big) \\
			& = \sum_{j=i_*(t)}^k m_j v_j(s) + \kappa s \sum_{j=i_*(t)}^k \sum_{\ell=1}^N m_j m_\ell \Phi(x_j(s) - x_\ell(s)) - \kappa s \big( \theta_k^2 - \theta_{i_*(t) - 1}^2 \big).
		\end{align*}
		Since the last two terms are continuous with respect to $s$, we have 
		\begin{align*} 
			f(s) - f(s-) 
			& = \sum_{j=i_*(t)}^k m_j v_j(s) - \sum_{j=i_*(t)}^k m_j v_j(s-)
			= \sum_{j=k_*(s)}^k m_j v_j(s) - \sum_{j=k_*(s)}^k m_j v_j(s-)
			\\
			& = (\theta_k - \theta_{k_*(s) - 1}) \bigg( v_k(s) - \frac{\sum_{j=k_*(s)}^k m_j v_j(s-)}{\sum_{j=k_*(s)}^k m_j}\bigg) \le 0,
		\end{align*} 
		where in the last inequality, we have used the sticky particle collision property:
		\[
		v_{k_*(s)}(s-) \ge v_{k_*(s) + 1}(s-) \ge \cdots \ge v_{k^*(s)-1}(s-) \ge v_{k^*(s)}(s-).
		\] 
	\end{proof}

	\section{Attractive regime: existence and approximability}
	\label{sec:gendata}
	
	In this section, we finish the existence part of the proof of Theorem \ref{thm:M} in the attractive case $\kappa<0$, by taking limits of the discrete entropy solutions generated in Section \ref{sec:discrete}.  Theorem~\ref{thm:discreteattractive} shows that all shocks in $M_N$ are admissible. As a consequence, the front-tracking scheme coincides with the sticky particle approximation. We briefly outline this correspondence here; the argument is largely analogous to that in \cite{leslie2023sticky}.
	
	Given initial data $M^0$, we first construct a sequence of piecewise constant approximations $M_N^0$ of the form \eqref{eq:MN0}, namely
	\[
	M_N^0(x) = -\frac12 + \sum_{i=1}^{N} m_{i,N} H(x - x_{i,N}^0).
	\]
	The flux $\A_N$ is given by \eqref{def:AN} where $A_N$ is the piecewise linear function (see \eqref{eq:AN}) such that
	\[
	A_N(\theta_{i,N}) = A(\theta_{i,N}),\quad \forall\,i=0,1,\ldots,N.
	\]
	The approximation can be constructed such that
	\[
	\lim_{N\to\infty}\|M_N^0-M^0\|_{L^1(\R)}=0,\quad\text{and}\quad
	\lim_{N\to\infty}\|A_N-A\|_{L^\infty([-\frac12,\frac12])}=0.
	\]
	Then, according to \eqref{eq:psi0}, we define the initial velocities 
	\[
	v_{i,N}^0 = \psi_{i,N}^0 - \sum_{j=1}^{N} m_{j,N} \Phi(x_{i,N}^0 - x_{j,N}^0)
	= \frac{A(\theta_{i,N})-A(\theta_{i-1,N})}{m_{i,N}} - \sum_{j=1}^{N} m_{j,N} \Phi(x_{i,N}^0 - x_{j,N}^0).
	\]
	The approximated solution $M_N$ is constructed by evolving the sticky particle dynamics \eqref{eq:ABS} with initial data $(m_{i,N},x_{i,N}^0,v_{i,N}^0)_{i=1}^N$.
	
	Fix a time $T>0$. For any $t\in[0,T]$, since $M_N(t)$ is uniformly bounded and nondecreasing, we apply Helly's selection theorem and find a convergent subsequence $M_{N_k}(t)$ in $L^1_{\loc}(\R)$. Using a diagonal argument, we can get a further subsequence, still denoted by $M_{N_k}$, that is convergent for all rational $t\in[0,T]$ in $L^1_{\loc}(\R)$.  We provisionally denote the limit by $M(t)$.
	
	Observe the uniform-in-$N$ bound $\{x_{i,N}^0\}_{i=1}^N\subset [-R^0, R^0]$.  Combining this with the bound $|\psi_{i,N}^0|\le |A_N|_{\Lip} \le |A|_{\Lip}$, the monotonicity of $\Phi$, and the fact that the total mass of all agents is $1$, we obtain 
	\[
	|v_{i,N}^0| \le |\psi_{i,N}^0| + \sum_{j=1}^N m_{j,N} |\Phi(x_{i,N}^0 - x_{j,N}^0)| \le |A|_{\Lip}  + \Phi(2R^0) \triangleq V^0,\quad \forall\, i=1,\ldots,N.
	\]
	We apply Proposition \ref{prop:xvbound} and obtain 
	\[
	\{x_{i,N}(t)\}_{i=1}^N\subset [-R(T),R(T)],\quad R(T) \triangleq R^0 + V^0T +\tfrac12|\kappa| T^2.
	\]
	It follows that $M_N(\pm x,t)=\pm\tfrac12$ for all $x>R(T)$ and $t\in[0,T]$. Consequently, we have $M_{N_k}(t)-M(t)\to 0$ in $L^1(\R)$ for all rational times $t\in \mathbb{Q}_+$.
	
	The convergence result can be extended to irrational times thanks to the time regularity estimate
	\begin{equation}\label{eq:timereg}
		\int_\R |M_N(x,t) - M_N(x,s)|\dd x \le \sum_{i=1}^Nm_{i,N}|x_{i,N}(t)-x_{i,N}(s)| \le V(T)(t-s).
	\end{equation}
	According to Proposition \ref{prop:xvbound}, $V(T) = |A|_{\Lip} +  \Phi(2R^0) + |\kappa| T$, which is independent of $N$.
	
	We conclude with 
	\[
	M_{N_k}-M\to 0\quad \text{in}\,\, C([0,T];L^1(\R)).
	\]
	Moreover, from \eqref{eq:timereg} we have the uniform bound $\|\pa_tM_{N_k}(\cdot,t)\|_{\mathcal{M}}\leq V(T)$ for any $t\in[0,T]$. Hence, extracting a further subsequence, still denoted by $M_{N_k}$, we obtain the weak-$*$ convergence
	\[\pa_tM_{N_k}(\cdot,t)\stackrel{*}{\rightharpoonup}  \pa_tM(\cdot,t)\quad \text{in}\,\, \mathcal{M}(\R).\] 
	This also allows us to conclude that $M\in BV(\R\times[0,T])$.
	
	Finally, we show that the function $M$ we have constructed above is indeed an entropy solution of \eqref{eq:M}.  We do this by verifying the entropy inequality \eqref{eq:entropydist} for all Kruzkov entropy pairs $(\eta, q)$ in \eqref{eq:Kruzkovpair}.
	
	We know from Theorem \ref{thm:discreteattractive} that $M_N$ is an entropy solution of \eqref{eq:MN}. Thus the entropy inequality \eqref{eq:entropydist} is satisfied for $(\eta, q_N)$:
	\[
	\int_0^{T} \int_\R \Big[\eta(M_N) \pa_t \zeta + q_N(M_N,t)\pa_x \zeta +
	(\phi\ast M_N)\zeta \pa_x(\eta(M_N)) \Big] \dd x \dd t
	\ge 0,
	\]
	where $\eta(m)=|m-\a|$ and $q_N(m,t) = \sgn{m - \alpha}\big(\A_N(m,t) - \A_N(\alpha,t)\big)$.
	Now we pass to the limit. Define $q(m,t) = \sgn{m-\a}(\A(m,t) -
	\A(\alpha,t))$, as in \eqref{eq:Kruzkovpair}. We get
	\begin{align*}
		\|\eta(M_N)-\eta(M)\|_{L^1}\leq
		&\,|\eta|_{\Lip}\|M_N-M\|_{L^1}=\|M_N-M\|_{L^1}\to0,\\
		\|q_N(M_N,t)-q(M,t)\|_{L^1}\leq &\,\|q_N(M_N,t)-q(M_N,t)\|_{L^1}+\|q(M_N,t)-q(M,t)\|_{L^1}\\
		\leq &\, 2R(T)\|A_N-A\|_{L^\infty([-\frac12,\frac12])}+\|\A\|_{\Lip(T)}\|M_N-M\|_{L^1}\to0.
	\end{align*}
	The convergence of the last term follows the same argument as in \cite{leslie2023sticky}. We omit the details here.
	
	We conclude that $M$ is an entropy solution of \eqref{eq:M}, and it can be approximated by $M_N$, generated from the sticky particle dynamics \eqref{eq:ABS}. 
	
	\section{Repulsive regime}\label{sec:repulsive}
	In this section, we consider the EPA system \eqref{eq:EPA} with a repulsive Poisson force, namely $\kappa>0$. In contrast to the attractive case, the entropy solution is \emph{not} consistent with the sticky particle approximation.
	
	\begin{theorem}\label{thm:discreterepulsive}
		Let $\kappa>0$. Under the assumptions in Theorem \ref{thm:discrete}, the sticky particle solution $M_N$ is not the entropy solution of the discretized balance law \eqref{eq:MN}.	
	\end{theorem}
	\begin{proof}
		We argue that the Oleinik entropy condition \eqref{eq:verify} fails to be satisfied.
		
		Take a small time $t>0$ such that no collision occurs yet, namely $J_i(t)=\{i\}$. Then, \eqref{eq:verify} implies
		\[
		\psi_i(t) \leq \frac{(\theta-\theta_{i-1})\psi_i^0}{\theta - \theta_{i-1}} - \kappa t (\theta_{i}-\theta) = \psi_i^0 - \kappa t (\theta_{i}-\theta) < \psi_i^0,\quad \forall\, \theta\in(\theta_{i-1}, \theta_i).
		\]	
		However, from \eqref{eq:psii0} we know $\psi_i(t) = \psi_i^0$. This leads to a contradiction. Therefore, $M_N$ is not an entropy solution of \eqref{eq:MN}.
	\end{proof}
	
	Theorem~\ref{thm:discreterepulsive} shows that the entropy solution cannot tolerate jump discontinuities. To illustrate this phenomenon, consider the initial condition
	\[
	\rho^0(x) = \delta_0(x),\quad\text{and correspondingly,}\quad M^0(x) = -\frac12+H(x),	
	\]
	which represents a single point mass located at $x=0$. Assume that the initial velocity is $u^0(x)=0$, and that there is no alignment interaction, i.e. $\phi\equiv0$. In this case, the flux satisfies $\A(m,t) = \kappa t m^2$. 
	
	For this repulsive Euler-Poisson equation with atomic initial data, the associated scalar balance law \eqref{eq:M} reduces to
	\[
	\pa_tM+ \pa_x(\kappa t M^2)=0,\quad M^0(x) = 
	\begin{cases}
		-\frac12 & x<0\\ \frac12 & x\geq0.
	\end{cases}
	\]
	This Riemann problem admits an explicit entropy solution given by
	\[
	M(x,t) =  \begin{cases}
		-\frac12 & x<-\frac{\kappa t^2}{2}\\ \frac{x}{\kappa t^2}& -\frac{\kappa t^2}{2} \leq x< \frac{\kappa t^2}{2}\\ \frac12 & x\geq \frac{\kappa t^2}{2},
	\end{cases}
	\]
	which is a rarefaction wave. Consequently,
	\[
	\rho(x,t) = \pa_xM(x,t) = \frac{1}{\kappa t^2} \mathbf{1}_{(-\frac{\kappa t^2}{2},\frac{\kappa t^2}{2})}(x),
	\]
	so the density is no longer a point mass. Thus, the repulsive interaction breaks the initial atom into infinitesimal pieces that spread out in space.
	
	We also remark that the solution may develop mass concentration (i.e.\ jump discontinuities in $M$) at finite time, even if the initial data are smooth; see, for instance, \cite{bhatnagar2023critical}. However, such concentrations are transient. Since the flux $\A(m,t)=A(m)+\kappa t m^2$ is strictly convex for sufficiently large $t$, all jump discontinuities must eventually disappear by the argument above.
	
	Regarding existence, the front-tracking scheme employs small $\delta$-shocks to approximate rarefaction waves. This procedure applies directly to \eqref{eq:M} and yields an entropy solution, but the resulting approximation no longer coincides with the sticky particle solution.
	
	Whether this entropy solution is physically relevant remains debatable. Other weak solutions exist, including the sticky particle solution \eqref{eq:MNsoln}. Another alternative was introduced in \cite{carrillo2023equivalence}, where the solution $M_N$ remains piecewise constant but an \emph{unsticking} mechanism is permitted. It is unclear whether such solutions enjoy uniqueness or stability properties comparable to those of the entropy solution constructed here. We refer the reader to \cite{carrillo2023equivalence} for further discussion in the context of the Euler-Poisson equations. Many of the arguments therein can be adapted to the EPA system \eqref{eq:EPA}.

\end{document}